\documentclass[12pt,reqno,a4paper]{amsart}
\usepackage{amsmath,amssymb,amsfonts,amscd}
\usepackage[mathscr]{eucal}
\usepackage{bbm}
\usepackage{tikz-cd}
\usetikzlibrary{arrows}
\usepackage{comment,mathabx}

\date{6 (19) February 2024}

\def\co{\colon\thinspace}

\newtheorem{theorem}{Theorem}
\newtheorem{lemma}{Lemma}
\newtheorem{proposition}{Proposition}
\newtheorem{corollary}{Corollary}

\theoremstyle{definition}
\newtheorem{definition}{Definition}
\newtheorem{example}{Example}

\newcommand{\K}{\mathbbm{k}}

\newcommand{\HH}{\mathbb{H}}

\newcommand{\la}{{\lambda}}

\newcommand{\der}[2]{{\frac{\partial {#1}}{\partial {#2}}}}

\DeclareMathOperator{\id}{id}
\newcommand{\RR}{\mathbb R}
\newcommand{\CC}{\mathbb C}

\newcommand{\al}{{\alpha}}
\newcommand{\be}{{\beta}}
\newcommand{\G}{{\Gamma}}

\newcommand{\e}{{\varepsilon}}

\renewcommand{\O}{{\Omega}}
\renewcommand{\o}{{\omega}}

\title[The weighted projective superspace with weights $+1, -1$]{The weighted projective superspace with weights $+1, -1$ and an analog of the Fubini--Study form}
\author{Ekaterina~Shemyakova}
\address{Department of Mathematics,  University of Toledo, Toledo,  Ohio, 43606, USA}
\email{ekaterina.shemyakova@utoledo.edu}
\author{Theodore~Voronov}
\address{Department of Mathematics,  University of Manchester, Manchester, M13 9PL,  UK}
\email{theodore.voronov@manchester.ac.uk}

\thanks{Research of the first author was partially supported by the NSF under grant  1708033 and by Simons Collaboration Grant. Research of the second author was partially supported by London Mathematical Society grants.}

\begin{document}
\begin{abstract}
As a by-product of our work on   super  Pl\"{u}cker embedding, we came to the notion of a weighted projective superspace $P_{+1,-1}(V\oplus W)$ with weights $+1,-1$. 
The construction is not in itself super and   makes sense in  ordinary  (purely even) framework.
Unlike  the  familiar  weighted projective spaces with positive weights, the (super)space $P_{+1,-1}(V\oplus W)$  is a  smooth (super)manifold. We describe its structure   and show that it possesses  an analog of the Fubini--Study form.
\end{abstract}

\maketitle
\tableofcontents


\section{Introduction}
This note is a by-product of our work on the  analog of the Pl\"{u}cker embedding for super Grassmannians~\cite{tv:pluecker} and \cite{shemyakova:2023}. For   general $G_{r|s}(n|m)$, with $r>0$ and $s>0$, the target of this embedding cannot be a usual projective superspace. We found that, instead, it has to be a particular analog of weighted projective spaces, namely, a ``weighted projective superspace''  with weights $1$ and $-1$. Our notation: $P_{+1,-1}(V\oplus W)$, where $V$ and $W$ are two given  vector superspaces. 

These     weights,  $\pm 1$,   come about from the homogeneity properties of   Berezinian and for our purpose were attached to what we introduced as ``super Plucker coordinates'' for a super Grassmannian. However, the construction of $P_{+1,-1}(V\oplus W)$  is independent of that purpose and is not peculiar for   supergeometric framework.

Weighted projective spaces with positive weights are well-known objects in topology and algebraic geometry. They are typically singular spaces, not manifolds. Unlike that, we show that the weighted projective (super)space $P_{+1,-1}(V\oplus W)$ is a smooth (super)manifold. Actually, it admits a nice description in terms of classical bundles. 

Our main results are: the  description   of $P_{+1,-1}(V\oplus W)$ via the tautological bundle over ordinary projective (super)spaces and a construction of a symplectic structure on $P_{+1,-1}(V\oplus W)$ analogous to the Fubini--Study form on $\CC P^n$.

Terminological remark: we often suppress the prefix `super-' writing e.g. `space' instead of `superspace' unless it is needed for emphasis.

\section{Definition and structure}
Let $V$ and $W$ be vector superspaces, which we regard as supermanifolds. (The reader not interested in supergeometry can think about usual vector spaces, manifolds, etc.) We consider vector spaces over a field $\K$, where $\K$ may be $\RR$ or $\CC$. (Some constructions   make sense for arbitrary $\K$. One can also consider the field of quaternions $\K=\HH$.) 

Important notation: for a vector (super)space $V$ or a (super) vector bundle $E$, by $V_0$ or $E_0$ respectively we denote the open subspace of all non-zero vectors. (In the supercase, this should not be confused with   parity;    points  of a vector superspace considered as a supermanifold are even vectors, so no need to indicate their parity.)

Consider the group $G=\K^*=\K_0$. Define its action on $V_0\times W_0$ by

\begin{equation}
    \la\cdot(v,w):=(\la  v,\la^{-1}w)\,,
\end{equation}
where $\la\in G$, $v\in V_0$ and $w\in W_0$.
\begin{definition} The \emph{weighted projective  space $P_{+1,-1}(V\oplus W)$ with weights $+1$ and $-1$} is defined as  the quotient
\begin{equation}
    P_{+1,-1}(V\oplus W):=(V_0\times W_0) /G
\end{equation}
by the above action of $G$.
\end{definition}

In the degenerate case $W=0$ or $V=0$ we obtain the usual projective spaces $P(V)$ or $P(W)$. Here and in the sequel $P(V)$ stands for the projectivization of a vector space $V$. We also denote by $E_V$ the tautological bundle over $P(V)$.

We denote points of $P_{+1,-1}(V\oplus W)$ as $[v\|w]$, where $v\in V_0$, $w\in W_0$,  and 
\begin{equation*}
  [v\|w]=[\la v\|\la^{-1}w]\,.
\end{equation*}
It follows that $[\la v\|w]=[\la v\|\la^{-1}(\la w)]=[v\|\la w]$\,. The pair $(v,w)\in V_0\times W_0$ can be regarded as \emph{homogeneous coordinates} of a point $[v\|w]. $

\begin{example} For $W=\K$, we have   $P_{+1,-1}(V\oplus \K)\cong V_0$. Note that $V_0$ can be identified with $(E_V)_0$ for the tautological bundle $E_V\to P(V)$.  
\end{example}

\begin{theorem} For arbitrary $V$ and $W$, the space $P_{+1,-1}(V\oplus W)$ has natural fiber bundle structures over $P(V)$ and $P(W)$\,:
  \begin{equation}
    \begin{tikzcd}
 & P_{+1,-1}(V\oplus W) \arrow[dl, "\pi_V" above left] \arrow[dr,"\pi_W"] \\
P(V)    && P(W)\,.
\end{tikzcd}
\end{equation}
The bundle $\pi_V\co P_{+1,-1}(V\oplus W)\to P(V)$ is isomorphic to the bundle $(E_V\otimes W)_0\to P(V)$ and the bundle $\pi_W\co P_{+1,-1}(V\oplus W)\to P(W)$ is isomorphic to   $(E_W\otimes V)_0\to P(W)$\,.
\end{theorem}
\begin{proof}
  A point $[v\|w]\in P_{+1,-1}(V\oplus W)$ is mapped to $v\otimes w$, which can be interpreted either as a vector in $E_V\otimes W$, and because $w\neq 0$, it is a non-zero vector, i.e. a point of $(E_V\otimes W)_0$, or, similarly with the roles of $V$ and $W$ swapped,  as  a non-zero vector in $V\otimes E_W$. 
\end{proof}

\begin{corollary}
There are two canonical atlases on $P_{+1,-1}(V\oplus W)$ obtained by lifting of the standard affine atlases for $P(V)$ and $P(W)$. 
\end{corollary}
Coordinates in each of these atlases are \emph{inhomogeneous coordinates} on $P_{+1,-1}(V\oplus W)$ (which   therefore are of two types). 

\section{Analog of Fubini--Study form}

In this section, $\K=\CC$. We want to construct an analog for our space $P_{+1,-1}(V\oplus W)$ of the classical Fubini--Study form on $P(V)$, i.e. $\CC P^n$  or      $\CC P^{n|m}$  in the supercase, see e.g.~\cite{tv:volumes}. 

Recall that for a complex vector (super)space $V$, the Fubini--Study form $\o$ on the projective space $P(V)$  is obtained by Hamiltonian reduction from the   symplectic  form $\o_0$ on $V$ with respect to the action of the group $S^1$. In brief, the procedure is as follows. Assuming $V$ is equipped with a Hermitian inner product, the form $\o_0$ is   as its imaginary part (up to a factor), while the real part gives a Euclidean structure. The unit sphere $S(V)\subset V$ arises as a  level set of the Hamiltonian generating the $S^1$-action on $V$. 
The quotient of $S(V)$ by the action of $S^1$ gives $P(V)$. The form $\o_0$ restricted on $S(V)$ is invariant under the action of the supergroup $\widehat{S^1}=\Pi TS^1$, hence descends on $P(V)$ giving the form   $\o$. 

In our case, we need to choose a symplectic structure on $V\oplus W$. The simplest choice would be
\begin{equation}\label{eq.omzero}
  \o_0=\frac{i}{2}(dvd\widebar v +dw d\widebar w)
\end{equation}
Here and in the sequel we write all the formulas as if we were in the purely even case, where the  expression  such as  $dvd\widebar v$ stands  for  the  sum   $dvd\widebar v=\sum dv^ad\widebar v^a$. (In the supercase, it should be $\sum_{\text{$a$ even}} dv^ad\widebar v^a- i \sum_{\text{$a$ odd}} dv^ad\widebar v^a$, see~\cite{tv:volumes}.) In principle, one can consider 
\begin{equation}\label{eq.omzeroalfa}
  \o_0=\frac{i}{2}(dvd\widebar v +\al dw d\widebar w)
\end{equation}
with $\al\in\RR$ instead. By rescaling, it is clear that only a choice between $\al=1$ and $\al=-1$ could make a difference. Using provisionally the symplectic structure depending on parameter $\al$, consider the Hamiltonian corresponding to the action of $S^1$. We have $(v,w)\mapsto (e^{it}v, e^{-it}w)$, so infinitesimally $(v,w)\mapsto (v+i\e v, w-i\e w)$, i.e. we have the vector field
\begin{equation*}
  X=iv^a\der{}{v^a}-i\bar v^a\der{}{\widebar{v^a}}- iw^k\der{}{w^k}-i\bar w^k\der{}{\widebar{w^k}}\,.
\end{equation*}
The corresponding Hamiltonian $H$ is found from 
\begin{multline*}
  -dH=\iota_X\o_0=\left(iv^a\der{}{dv^a}-i\bar v^a\der{}{d\bar v^a}- iw^k\der{}{dw^k}-i\bar w^k\der{}{d\widebar w^k}\right)\o_0=\\
  -\frac{1}{2}\left(\sum (v^ad\widebar{v^a}+\widebar{v^a}dv^a)-\al \sum (w^k d\widebar{w^k}+ dw^k \widebar{w^k})\right)=
  -\frac{1}{2}d\left(\sum v^a \widebar{v^a} -\al \sum w^k \widebar{w^k}\right)\,.
\end{multline*}
Hence we can take $H=\frac{1}{2}\left(\sum v^a \widebar{v^a} -\al \sum w^k \widebar{w^k}\right)$. 

\begin{proposition}
The level set $\G_{\al,\be}:=\{H=\be\}$ for arbitrary $\be\in\RR$ has intersection with every orbit of the action of $\CC^*$ if and only if $\al>0$.  
\end{proposition}
\begin{proof} For arbitrary $(v,w)\in V_0\times W_0$ we need to find $\la\in \CC^*$ so that $\la(v,w)\in \G_{\al,\be}$. This boils down to a 
quadratic equation for $|\la|^2$ \,:
\begin{equation*}
  (v\widebar{v})|\la|^4-\be|\la|^2-\al(w\widebar{w})=0\,.
\end{equation*}
The discriminant is $\be^2+4\al (v\widebar{v})(w\widebar{w})$ is non-negative for all $v,w$ only if $\al>0$, and in this case it is $\geq \be^2$, hence there is always a positive root.
\end{proof}

So as explained we can set $\al=1$, which gives~\eqref{eq.omzero} as the symplectic form on $V_0\times W_0$. Define $\G_{\be}:=\G_{1,\be}$,
\begin{equation*}
  \frac{1}{2}\left(\sum v^a \widebar{v^a} - \sum w^k \widebar{w^k}\right)=\be\,.
\end{equation*}
Essentially there are the three cases:  $\G_{-1}$, $\G_{0}$ and $\G_{+1}$. We will consider the simplest, which is  $\G_{0}$,
\begin{equation*}
  \G_{0}=\left\{ \sum v^a \widebar{v^a} = \sum w^k \widebar{w^k}\right\}\subset V_0\times W_0\,.
\end{equation*}
Every $[v\|w]\in P_{1,-1}(V\oplus W)$ has a representative in $\G_0$, 
\begin{equation}\label{eq.r}
  (\la v,\la^{-1}w)\in \G_0 \quad\text{for} \quad  \la =\left(\frac{w\widebar{w}}{v\widebar{v}}\right)^{1/4}\,,
\end{equation}
and we arrive at the commutative diagram
\begin{equation}
    \begin{tikzcd}
    \G_0 \arrow[rr] \arrow[dr, "S^1" below left, "p" above right] & & V_0\times W_0 \arrow[dl, "\CC^*" below right, "q" above left]\\
    & P_{1,-1}(V\oplus W)
\end{tikzcd}
\end{equation}
where the left arrow is an analog of the Hopf fibration. 

\begin{lemma}
The restriction of the form $\o_0$ onto $\G_0$ is invariant under the action of the supergroup $\widehat{S^1}=\Pi TS^1$.
\end{lemma}
\begin{proof}
  The action of $\widehat{S^1}=\Pi TS^1$ is the lifting of the action $(v,w)\mapsto (e^{it}v,e^{-it}w)$, so that $(dv,dw)\mapsto (idte^{it}v+e^{it}dv, -idt e^{-it}w+e^{-it}w)=\bigl(e^{it}(idtv+dv), e^{-it} (-idtw+dw)\bigr)$. For the pullback of $\o_0$ by this transformation we obtain that after simplification  
  \begin{equation*}
    (e^{it},d\,e^{it})^*\left(dvd\widebar v +dw d\widebar w\right)=dvd\widebar v +dw d\widebar w +idt\,d(v\widebar v - w  \widebar w)=dvd\widebar v +dw d\widebar w
  \end{equation*}
  on $\G_0$ (actually, on any level set  $\G_{\be}$). 
\end{proof}

It follows, exactly as in the familiar situation of the projective superspace $P(V)$ such as $\CC P^n$ or $\CC P^{n|m}$, that there is a uniquely defined closed form $\o\in \O^2(P_{1,-1}(V\oplus W))$ whose pullback on $\G_0$ is the restriction of $\o_0$. It remains to give for it an explicit expression in local coordinates or its pullback on $V_0\times W_0$ (= expression in terms of homogeneous coordinates on $P_{1,-1}(V\oplus W)$). 

\begin{theorem}[symplectic form on  $P_{1,-1}(V\oplus W$] \label{thm.form}
The form $\o\in \O^2(P_{1,-1}(V\oplus W))$ is given by the following expression:
\begin{multline}\label{eq.omega}
  \o=\frac{i}{2}(v\widebar{v})^{-3/2} (w\widebar{w})^{-3/2}
  \left((v\widebar{v})(w\widebar{w})^2(dvd\widebar{v})+ (v\widebar{v})^2(w\widebar{w}) (dwd\widebar{w})  - \frac{1}{2} (w\widebar{w})^2(dv \widebar{v})(vd\widebar{v}) \right. \\
   \left.  - \frac{1}{2} (v\widebar{v})^2(dw \widebar{w})(wd\widebar{w})  + \frac{1}{2} (v\widebar{v}) (w\widebar{w})(dw \widebar{w})(\widebar{v} v)
  + \frac{1}{2} (v\widebar{v}) (w\widebar{w})(dv \widebar{v})(d\widebar{w} w)%
  \right)
\end{multline}
  where $v,w$ are regarded as homogeneous coordinates on $P_{1,-1}(V\oplus W)$.
\end{theorem}

To obtain the expression in inhomogeneous coordinates (in either of the two canonical atlases), one can take formula~\eqref{eq.omega} and substitute $v^{a_0}=1, dv^{a_0}=0$ for a chosen even $v^{a_0}$ or  do the same with some even $w^{i_0}$. Here $a_0$ or $i_0$ is the number of an affine chart in one of the atlases. The remaining variables shall be regarded as inhomogeneous coordinates and the form of the resulting expression will be the same as~\eqref{eq.omega}.

\begin{proof}[Proof of Theorem~\ref{thm.form}]
Consider again the above diagram where we denoted by $i$ the inclusion map for $\G_0$ and introduced   $r$, which is the projection onto $\G_0$ given by~\eqref{eq.r}\,:
\begin{equation}
    \begin{tikzcd}
    \G_0 \arrow[shift left]{rr}{i} \arrow[dr, "p" below left] & & V_0\times W_0 \arrow[dl, "q" below right] \arrow[shift left]{ll}{r}\\
    & P_{1,-1}(V\oplus W)
\end{tikzcd}
\end{equation}
We have $r\circ i=\id$,  $p\circ r=q$,  and   $p^*\o=i^*\o_0$; hence $q^*\o=r^*p^*\o=r^*i^*\o_0=(i\circ r)^*\o_0$. Denote $s:=i\circ r$, $s\co V_0\times W_0\to  V_0\times W_0$. Explicitly,  
\begin{equation*}
  s\co (v,w)\mapsto  (\la v,\la^{-1}w) \quad\text{where} \quad  \la =\left(\frac{w\widebar{w}}{v\widebar{v}}\right)^{1/4}\,.
\end{equation*}
To obtain the desired expression for $\o$ in homogeneous coordinates, which is $q^*\o$, we need to calculate $s^*\o_0$. This is a lengthy, but straightforward calculation, which gives formula~\eqref{eq.omega}
\end{proof}

It would be interesting to consider similarly the Hamiltonian reduction with the help of $\G_{-1}$ and  $\G_{+1}$. 


\begin{thebibliography}{1}

\bibitem{shemyakova:2023}
Ekaterina Shemyakova.
\newblock On super cluster algebras based on super {P}l\"{u}cker and super
  {P}tolemy relations.
\newblock {\em J. Geom. Phys.}, 188:Paper No. 104776, 16, 2023.

\bibitem{tv:pluecker}
Ekaterina Shemyakova and Theodore Voronov.
\newblock On super {P}l\"{u}cker embedding and cluster algebras.
\newblock {\em Selecta Math. (N.S.)}, 28(2):Paper No. 39, 58, 2022.

\bibitem{tv:volumes}
Theodore Voronov.
\newblock On volumes of classical supermanifolds.
\newblock {\em Sbornik: Mathematics}, 207(11):1512--1536, 2016.

\end{thebibliography}
\def\cprime{$'$} 

\end{document}